\documentclass[12pt]{amsart}
\usepackage[colorlinks=true,citecolor=black,linkcolor=black,urlcolor=blue]{hyperref}
\usepackage{amsmath}
\usepackage[english,  activeacute]{babel}
\usepackage[utf8]{inputenc}
\usepackage{amssymb}
\usepackage{amsthm}
\usepackage{graphics,graphicx}
\usepackage{enumerate}
\usepackage{ytableau}
\usepackage{array}
\usepackage{bm}
\usepackage{cite}
\usepackage{a4wide}
\usepackage{color, url}
\usepackage{float}
\usepackage{comment}
\usepackage{xcolor}
\usepackage{color, url}
\usepackage{float}
\usepackage{accents}
\usepackage{pgf, tikz}
\usepackage{multirow}
\newcommand{\seqnum}[1]{\href{http://oeis.org/#1}{\underline{#1}}}

\theoremstyle{plain}
\newtheorem{theorem}{Theorem}[section]

\newtheorem{corollary}[theorem]{Corollary}

\newtheorem{remark}[theorem]{Remark}

\theoremstyle{definition}
\newtheorem{definition}[theorem]{Definition}

\newtheorem{openproblem}{Open Problem}[section]

\newcommand{\Sym}{{\mathfrak S}}

\newcommand{\I}{{\mathcal{I}}}

\newcommand{\maj}{{\sf{maj}}}
\newcommand{\inv}{{\sf{inv}}}
\newcommand{\D}{{\mathcal D}}
\newcommand{\col}{{\sf{col}}}

\usepackage[usestackEOL]{stackengine}[2013-10-15]
\stackMath

\title{Inversions in Colored Permutations, Derangements, and Involutions}

\date{\today}

\begin{document}

\author[M. Ahmia]{Moussa Ahmia}
\address{University of Mohamed Seddik Benyahia, LMAM laboratory, Jijel, Algeria}
\email{moussa.ahmia@univ-jijel.dz;
ahmiamoussa@gmail.com}

\author[J. L. Ram\'{\i}rez]{Jos\'e L. Ram\'{\i}rez}
\address{Departamento de Matem\'aticas,  Universidad Nacional de Colombia,  Bogot\'a, Colombia}
\email{jlramirezr@unal.edu.co}

\author[D. Villamizar]{Diego Villamizar}
\address{Department of Mathematics, Xavier University of Louisiana, New Orleans LA, USA}
\email{dvillami@xula.edu}
\newcommand{\nadji}[1]{\mbox{}{\sf\color{green}[Ram\'{\i}rez: #1]}\marginpar{\color{green}\Large$*$}}

\begin{abstract}
Arslan, Altoum, and Zaarour introduced an inversion statistic for generalized symmetric groups \cite{Ars2}. In this work, we study the distribution of this statistic over colored permutations, including derangements and involutions. By establishing a bijective correspondence between colored permutations and colored Lehmer codes, we develop a unified framework for enumerating colored Mahonian numbers and analyzing their combinatorial properties. We derive explicit formulas, recurrence relations, and generating functions for the number of inversions in these families, extending classical results to the colored setting. We conclude with explicit expressions for inversions in colored derangements and involutions.
\end{abstract}

\maketitle
\section{Introduction}

Let $\Sym_n$ denote the set of all permutations of the set  $[n]:=\{1, 2, \dots, n\}$. Two classical statistics associated with a permutation  $\pi=\pi_1 \pi_2\cdots \pi_n  \in \Sym_n$ are the \emph{major index} and \emph{inversion number}, defined respectively by
\begin{align*}
    \maj(\pi):=\sum_{\pi_i>\pi_{i+1}}i \quad \text{and} \quad  \inv(\pi):=|\{(i,j): i<j \text{ and } \pi_i>\pi_j\}|.
\end{align*}
For example, for the permutation $\pi=2314$, we have 
$\maj(2314)=2$ and $\inv(2314)=2$. Throughout this work, we use one-line notation for permutations; that is, the permutation  $\pi = \left(\begin{smallmatrix} 1 & 2 & 3 & 4  \\
2 & 3 & 1 & 4 \end{smallmatrix}\right)$
in two-line notation is written simply as $\pi = 2314$.

A classical result due to MacMahon~\cite{Mac, Mac2} establishes that the major index and inversion number are equidistributed over $\Sym_n$. In other words, 
\begin{align}\label{MacMahon}
    \sum_{\pi \in \Sym_n} q^{\maj(\pi)} =
    \sum_{\pi \in \Sym_n} q^{\inv(\pi)} = [1]_q[2]_q\cdots [n]_q,
\end{align}
where $[i]_q:=1+q+\cdots + q^{i-1}$ denotes the $q$-analogue of the integer $i$.

Let $\I(n,k)$ denote the set of permutations of $[n]$ with exactly $k$ inversions, and let $i(n,k):=|\I(n,k)|$.  The sequence $i(n,k)$ is called the \emph{Mahonian numbers}.  For example, 
$$\I(4,2)=\{ 1342,\, 1423, \, 2143, \, 2314,\,  3124\}.$$

From \eqref{MacMahon}, it follows that the generating function of the Mahonian numbers is given by
$$\sum_{k=0}^{\binom{n}{2}}i(n,k)q^k=[1]_q[2]_q\cdots [n]_q.$$

For a positive integer $c$, consider the alphabet of $cn$ distinct symbols defined by
\[
\Sigma_{c,n} := \left\{ i^{[j]} : 1 \le i \le n,\ 0 \le j < c \right\}.
\]
An element of the form $i^{[j]}$ is called a \emph{colored element}. For small values of $j$, we adopt a shorthand notation using bars: specifically, $i^{[0]} := i$, $i^{[1]} := \overline{i}$, and $i^{[2]} := \overline{\overline{i}}$. By convention, we also define $i^{[c]} := i$ to ensure periodicity modulo $c$. Therefore, the set $\Sigma_{c,n}$ can be written explicitly as
\[
\Sigma_{c,n} = \left\{1,\, \overline{1},\, \overline{\overline{1}},\, \dots,\, 1^{[c-1]},\, 2,\, \overline{2},\, \overline{\overline{2}},\, \dots,\, 2^{[c-1]},\, \dots,\, n,\, \overline{n},\, \overline{\overline{n}},\, \dots,\, n^{[c-1]} \right\}.
\]

A \emph{colored permutation} $\sigma$ is a bijection on the set $\Sigma_{c,n}$ that satisfies   the condition 
\[
\sigma(i^{[j]}) = \left(\sigma(i)\right)^{[j]}
\quad \text{for all } i \in [n] \text{ and } 0 \le j < c.
\]
For example,  the following is a colored permutation of $\Sigma_{3,4}$:
 \begin{equation}
\label{pi-def}
\sigma = \begin{pmatrix} 1 & 2 & 3 & 4 & \overline{1} & \overline{2} & \overline{3} & \overline{4} & \overline{\overline{1}} &  \overline{\overline{2}} &  \overline{\overline{3}} &  \overline{\overline{4}} \\
\overline{3} & 2 & \overline{\overline{1}} & \overline{4} & \overline{\overline{3}} & \overline{2} & 1 & \overline{\overline{4}} & 3 & \overline{\overline{2}} & \overline{1} & 4
\end{pmatrix}.
\end{equation}
By omitting the top row, we obtain the \emph{one-line notation} for $\sigma$: 
\[
\sigma = \overline{3}\, 2\, \overline{\overline{1}}\, \overline{4}\, \overline{\overline{3}}\, \overline{2}\, 1\, \overline{\overline{4}}\, 3\, \overline{\overline{2}}\, \overline{1}\, 4.
\]

Let $G_{c,n}$ denote the set of colored permutations of $\Sigma_{c,n}$. Clearly, $|G_{c,n}|=c^nn!$.   We represent an element of $G_{c,n}$ by a word $\sigma_1^{\col(\sigma_1)}\sigma_2^{\col(\sigma_2)}\cdots \sigma_n^{\col(\sigma_n)}$, which corresponds to the first $n$ entries  in the one-line notation. Here, $\col(\sigma_i)$ denotes the color of the entry $\sigma_i$, that is, the superscript (or number of bars) of $\sigma_i$. For example, the colored permutation in \eqref{pi-def} is written as $\overline{3}\,2\, \,\overline{\overline{1}}\,\overline{4}$ and $\col(\sigma_1)=1, \col(\sigma_2)=0, \col(\sigma_3)=2$, and $\col(\sigma_4)=0$.

In  algebraic combinatorics,  the set  $G_{c,n}$ corresponds to the wreath product $C_c\wr \Sym_n$, where $C_c=\{0, 1, \dots, c-1\}$ is the cyclic group of order $c$, and $\Sym_n$ is the symmetric group on $[n]$.  While the group structure of $G_{c,n}$ will not play a role in this work, we point out that it generalizes several well-known families: for $c = 1$, we recover the symmetric group $\Sym_n$, which is the Coxeter group of type $A$; and for $c = 2$, we obtain the hyperoctahedral group, corresponding to the Coxeter group of type $B$.

Recently, Arslan, Altoum, and Zaarour~\cite{Ars2} introduced an inversion statistic on the set of colored permutations. Let $\sigma=\sigma_1^{\col(\sigma_1)}\sigma_2^{\col(\sigma_2)}\cdots \sigma_n^{\col(\sigma_n)}$ be a colored permutation in  $G_{c,n}$. We define the \emph{underlying permutation} of $\sigma$ as $|\sigma| := \sigma_1 \sigma_2 \cdots \sigma_n$, and the \emph{color statistic} as 
\[
\col(\sigma) := \col(\sigma_1) + \col(\sigma_2) + \cdots + \col(\sigma_n).
\]
A \emph{colored inversion} of $\sigma$ is then defined by
\begin{align}\label{maindef}
    \inv_c(\sigma):=\inv(|\sigma|) + \col(\sigma) + c \sum_{c_j \neq 0} |\{(i,j): i<j \text{ and } \sigma_i<\sigma_j\}|.
\end{align}
For example, for the colored permutation given in~\eqref{pi-def}, we have $\inv_c(\sigma) = 3 + 4 + 3 \cdot 3 = 16$. 

The colored permutation in $G_{c,n}$ that achieves the maximum number of colored inversions is $\sigma_{\max}:=1^{[c-1]}\,2^{[c-1]}\,\cdots n^{[c-1]}$, which satisfies $\inv_c(\sigma_{\max})=(c-1) n + c\binom{n}{2}$ inversions. Other statistics on colored permutations have also been studied; see, for instance,~\cite{Ste}.

Arslan, Altoum, and Zaarour~\cite{Ars2} proved the following extension of \eqref{MacMahon}: 
\begin{align}
    \label{qidentity1}
   \sum_{\sigma\in G_{c,n}}q^{\inv_c(\sigma)} = [c]_q[2c]_q\cdots [cn]_q.
\end{align}
In a similar vein, Bagno~\cite{Bagno} introduced a generalization of the major index and proved that its distribution over $G_{c,n}$ is also given by the identity \eqref{qidentity1}.

Consider the set of colored  permutations of $[n]$ with $k$ inversions, defined by
$$\I_c(n,k):=\{\sigma\in G_{c,n}: \inv_c(\sigma)=k\},$$ 
and let $i_c(n,k):=|\I_c(n,k)|$.  The sequence $i_c(n,k)$ is called the \emph{colored Mahonian numbers}.  For example, $i_3(4,2) = 10$, and the corresponding permutations are 
$$\I_3(4,2)=\{ 1342, \, 1423, \, 2143, \, 2314, \, 3124, \, \bar{2}134, \, \bar{1}243, \,\bar{1}324,  \,2\bar{1}34,  \, \overline{\overline{1}}234\}.$$

From \eqref{MacMahon}, it is clear the following generating function:
\begin{align}\label{gfcolr}
\sum_{k=0}^{(c-1) n + c\binom{n}{2}}i_c(n,k)q^k=[c]_q\cdots [cn]_q=\prod_{i=1}^{n}\frac{1-q^{ic}}{1-q}.
\end{align}

Let $I_{c,n}$ denote the total number of colored inversions in $G_{c,n}$. It is clear that 
\begin{equation*}
I_{c,n}=\sum_{\sigma\in G_{c,n}}\inv_c(\sigma)=\sum_{k=0}^{(c-1) n + c\binom{n}{2}}ki_c(n,k).
\end{equation*}

The outline of the paper is as follows. In Section~2, we introduce a colored variant of Lehmer codes and establish a bijection between these codes and  colored permutations in $G_{c,n}$. This construction provides a natural encoding of colored permutations, facilitating the analysis of their combinatorial properties. In Section~3, we use colored Lehmer codes to generalize several classical results related to the Mahonian numbers. Notably, our approach extends previous findings for type $A$ and $B$ permutations, offering a unified framework applicable to arbitrary colorings.  The Lehmer codes offer alternative proofs for known results.  Section~4 focuses on the total number of inversions in colored derangements. We derive explicit combinatorial identities to compute these.  Finally, in Section 5, we study the total number of inversion  on the set of colored inversion.  We obtain nontrivial formulas for these and provide generating functions.

\section{Colored Lehmer codes}
A \emph{Lehmer code} of length $n$ is an $n$-tuple $\ell = (\ell _1,\ell _2, \dots, \ell_n)$ such that $0\leq \ell_i<i$ for all $i=1, 2, \dots, n$. We denote by $\mathcal{L}_n$ the set of all Lehmer codes of length $n$. It is well known that  $\mathcal{L}_n$ is in bijection with the set of permutations of $[n]$ (cf.\, \cite{KN1}). This correspondence is given explicitly as follows. Given a permutation $\sigma \in \Sym_n$, its Lehmer code is defined by $$L(\sigma) = (\inv^{(1)}(\sigma),\inv^{(2)}(\sigma),\dots, \inv^{(n)}(\sigma)),$$
where 
\[\inv^{(i)}(\sigma)= |\{j \in [i-1]: j\text{ appears to the right of }i\text{ in }\sigma\}|.\]
For example, let $\sigma = 3214 \in \Sym_4$. The corresponding Lehmer code is $L(\sigma) = (0,1,2,0)$. Indeed, there are no elements smaller than $1$ to its right, so the first entry is $0$; there is one element smaller than $2$ to its right (namely $1$), so the second entry is $1$; there are two elements smaller than $3$ to its right ($1$ and $2$), yielding the third entry $2$; and finally, there are no elements smaller than $4$ to its right, so the last entry is $0$.

By construction, it is clear that $L(\sigma) \in \mathcal{L}_{n}$, and its entries refine the inversion set of $\sigma$. Moreover, the original permutation can be uniquely reconstructed from its Lehmer code by starting with an empty sequence and sequentially inserting the elements of $[n]$ one by one, placing the $i$-th number in the $\ell_i$-th position from the right.  For example, consider the code $(0,1,2,0) \in \mathcal{L}_4$.  The reconstruction process is illustrated in Figure~\ref{figLehmer}.
    \begin{figure}[h!]
        \centering
        \includegraphics[scale=0.7]{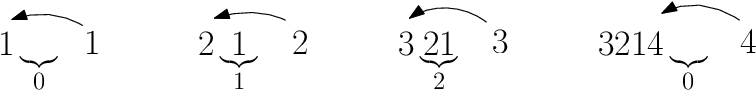}
        \caption{Reconstructing the permutation from its Lehmer code.}
        \label{figLehmer}
    \end{figure}
    
As observed above, the Lehmer code refines the inversion set in the sense that the sum of its entries equals the total number of inversions in the corresponding permutation. This motivates the definition of the following set:
$$\mathcal{L}_{n,k} = \{\ell \in \mathcal{L}_n: \ell _1+\ell _2+\cdots +\ell _n=k\},$$
which consists of all Lehmer codes of length $n$ whose entries sum to  $k$. By construction, it follows that   $\left | \mathcal{L}_{n,k}\right |= i(n,k)$. We now introduce the notion of \textit{colored Lehmer codes}.

\begin{definition}
    Let $n$ and $c$ be non-negative integers. A \emph{colored Lehmer code} is an $n$-tuple $\ell = (\ell _1,\dots,\ell _n)$ such that each entry satisfies $0\leq \ell _i <c\cdot i$. The set of colored Lehmer codes of length $n$ is denoted by $\mathcal{L}^{(c)}_{n}$.  Furthermore, the subset of colored Lehmer codes whose entries sum to a fixed integer $k$ is given by $$\mathcal{L}_{n,k}^{(c)} = \{\ell \in \mathcal{L}_n^{(c)}: \ell _1+\ell _2+\cdots +\ell _n=k\},$$
\end{definition}

For example, 
\begin{align*}
    \mathcal{L}_{4,5}^{(2)}=&\{0005,0014,0023,0032,0041,0050,0104,0113,0122,0131,014
   0, 0203,0212,0221,\\
   &0230,0302,0311,0320,1004,1013,1022,
   1031,1040,1103,1112,1121,1130,1202,\\
   &1211,1220,1301,1310\}.
\end{align*}

Observe that $|\mathcal{L}^{(c)}_{n}| = (c\cdot 1)\cdot (c\cdot 2)\cdots (c\cdot n) = c^nn!$, which suggests the existence of a bijection between  the set of colored permutations $G_{c,n}$ and the set of colored Lehmer codes $\mathcal{L}^{(c)}_{n}$. In what follows,  we show that this can be refined to colored Lehmer codes whose entries sum to $k$ and the colored Mahonian numbers. 
\begin{theorem}
\label{thmcollehcode}
    The number of colored Lehmer codes in $\mathcal{L}^{(c)}_{n,k}$ is given by $i_c(n,k)$.
\end{theorem}
\begin{proof}
    We begin by observing that
    \begin{align*}
        [c\cdot i]_q&= 1+q+\cdots +q^{c-1}+q^c+\cdots +q^{2c-1}+\cdots +q^{c\cdot i-1}\\
        &=(1+q+\cdots +q^{c-1})+q^c(1+q+\cdots +q^{c-1})+\cdots +q^{c(i-1)}(1+q+\cdots +q^{c-1})\\
        &=[c]_q\left (1+q^c+q^{2c}+\cdots +q^{c(i-1)}\right )\\
        &=[c]_q\cdot [i]_{q^c}.
    \end{align*}
Therefore,  the right-hand side of Equation~\eqref{qidentity1} can be expressed as
    $$   \sum_{\sigma\in G_{c,n}}q^{\inv_c(\sigma)} =[c]_q[2c]_q\cdots [nc]_q = [c]_q^n[n]_{q^c}!.$$
The coefficient of $q^b$ in the expansion of $[c]_q^n$ corresponds to the number of compositions (i.e., ordered integer partitions) $\bold{b} = (b_1,\dots, b_n)$ of $b$ into $n$ parts, where each part satisfies $0\leq b _i < c$ and $b_1+\cdots+b_n=b$. We denote this set of compositions by $\textbf{Com}_{n,b}^{<c}$.  Expanding the product $[c]_q^n[n]_{q^c}!$, we obtain
\begin{align*}
    [c]_q^n\cdot [n]_{q^c}! &= \left (1+q+\cdots + q^{c-1}\right )^n \left (\sum _{a= 0}^{\binom{n}{2}}i(n,a)q^{a\cdot c}\right )\\
    &=\sum _{k = 0}^{(c-1)n+c\binom{n}{2}}q^k \left (\sum _{c\cdot a+b = k}\left | \textbf{Com}_{n,b}^{<c}\right |\cdot \left | \mathcal{L}_{n,\frac{k-b}{c}}\right |\right ).
\end{align*}
The equality above  provides a formula for $i_c(n,k)$ in terms of $i(n,k)$. It remains to shown that this formula also counts the number of colored Lehmer codes. That is, we need to prove
$$\left |\mathcal{L}^{(c)}_{n,k}\right |= \sum _{c\cdot a+b = k}\left | \textbf{Com}_{n,b}^{<c}\right |\cdot \left | \mathcal{L}_{n,\frac{k-b}{c}}\right |.$$

To stablish this, we consider a decomposition of a colored Lehmer code.  Given $\ell = (\ell _1,\ell _2, \dots ,\ell _n) \in \mathcal{L}^{(c)}_{n,k}$,  the division algorithm allows us to write 
    \[
    \ell_i = c\cdot a_i + b_i, \quad \text{with } 0 \leq b_i < c.
    \]
Define  $\textbf{a}_{\ell} = (a_1,\dots ,a_n)$ and $\textbf{b}_{\ell} = (b_1,\dots,b_n)$. The condition  \(\ell_i = c\cdot a_i + b_i < c\cdot i\) implies $0\leq a_i<i$, so $\textbf{a}_{\ell} \in \mathcal{L}_{n,a}$. Similarly, since  $0 \leq b_i < c$, we also have $\textbf{b}_{\ell} \in \textbf{Com}_{n,b}^{<c}$.  This decomposition is reversible. Given $\textbf{a} \in \mathcal{L}_{n,a}$ and $\textbf{b} \in \textbf{Com}_{n,b}^{<c}$, we obtain $c\cdot \textbf{a}+\textbf{b} \in \mathcal{L}^{(c)}_{n,c\cdot a+b}$. The result follows by summing over all possible values of $a$ and $b$ such that $k = c\cdot a+b$.
\end{proof}

    The proof of Theorem \ref{thmcollehcode} uses a decomposition of Lehmer codes using the division algorithm. When starting with the colored Lehmer codes, this decomposition can be interpreted as colored permutations $\sigma \in G_{c,n}$, where  $\accentset{\kern-.4em\sim}{\inv_c}(\sigma) = c\cdot \inv (|\sigma|)+\col (\sigma) = k$. In this sense, the two statistics $\accentset{\kern-.4em\sim}{\inv_c}$ and ${\inv}_c$ have the same distribution in $G_{c,n}$. For example, consider the case $c = 2$ and $n = 3$. In this case, the distribution for both statistics is as follows:
    
    \begin{center}
  \begin{tabular}{ |c|c|c| } 
\hline
$k$ & $ {\inv}_c$ &  $\accentset{\kern-.4em\sim}{\inv_c}$\\
\hline
0 & $  1 2 3 $ & $  1 2 3 $ \\
1 & $  \overline{1} 2 3 \,  1 3 2 \,  2 1 3 $ & $  1 2 \overline{3} \,  1 \overline{2} 3 \,  \overline{1} 2 3 $ \\
2 & $  \overline{1} 3 2 \,  2 \overline{1} 3 \,  \overline{2} 1 3 \,  2 3 1 \,  3 1 2 $ & $  1 \overline{2} \overline{3} \,  \overline{1} 2 \overline{3} \,  \overline{1} \overline{2} 3 \,  1 3 2 \,  2 1 3 $ \\
3 & $  1 \overline{2} 3 \,  \overline{2} \overline{1} 3 \,  2 3 \overline{1} \,  \overline{2} 3 1 \,  3 \overline{1} 2 \,  \overline{3} 1 2 \,  3 2 1 $ & $  \overline{1} \overline{2} \overline{3} \,  1 3 \overline{2} \,  1 \overline{3} 2 \,  \overline{1} 3 2 \,  2 1 \overline{3} \,  2 \overline{1} 3 \,  \overline{2} 1 3 $ \\
4 & $  \overline{1} \overline{2} 3 \,  1 3 \overline{2} \,  1 \overline{3} 2 \,  \overline{2} 3 \overline{1} \,  \overline{3} \overline{1} 2 \,  3 2 \overline{1} \,  3 \overline{2} 1 \,  \overline{3} 2 1 $ & $  1 \overline{3} \overline{2} \,  \overline{1} 3 \overline{2} \,  \overline{1} \overline{3} 2 \,  2 \overline{1} \overline{3} \,  \overline{2} 1 \overline{3} \,  \overline{2} \overline{1} 3 \,  2 3 1 \,  3 1 2 $ \\
5 & $  1 2 \overline{3} \,  \overline{1} 3 \overline{2} \,  \overline{1} \overline{3} 2 \,  2 \overline{3} 1 \,  3 1 \overline{2} \,  3 \overline{2} \overline{1} \,  \overline{3} 2 \overline{1} \,  \overline{3} \overline{2} 1 $ & $  \overline{1} \overline{3} \overline{2} \,  \overline{2} \overline{1} \overline{3} \,  2 3 \overline{1} \,  2 \overline{3} 1 \,  \overline{2} 3 1 \,  3 1 \overline{2} \,  3 \overline{1} 2 \,  \overline{3} 1 2 $ \\
6 & $  \overline{1} 2 \overline{3} \,  2 1 \overline{3} \,  2 \overline{3} \overline{1} \,  \overline{2} \overline{3} 1 \,  3 \overline{1} \overline{2} \,  \overline{3} 1 \overline{2} \,  \overline{3} \overline{2} \overline{1} $ & $  2 \overline{3} \overline{1} \,  \overline{2} 3 \overline{1} \,  \overline{2} \overline{3} 1 \,  3 \overline{1} \overline{2} \,  \overline{3} 1 \overline{2} \,  \overline{3} \overline{1} 2 \,  3 2 1 $ \\
7 & $  1 \overline{3} \overline{2} \,  2 \overline{1} \overline{3} \,  \overline{2} 1 \overline{3} \,  \overline{2} \overline{3} \overline{1} \,  \overline{3} \overline{1} \overline{2} $ & $  \overline{2} \overline{3} \overline{1} \,  \overline{3} \overline{1} \overline{2} \,  3 2 \overline{1} \,  3 \overline{2} 1 \,  \overline{3} 2 1 $ \\
8 & $  1 \overline{2} \overline{3} \,  \overline{1} \overline{3} \overline{2} \,  \overline{2} \overline{1} \overline{3} $ & $  3 \overline{2} \overline{1} \,  \overline{3} 2 \overline{1} \,  \overline{3} \overline{2} 1 $ \\
9 & $  \overline{1} \overline{2} \overline{3} $ & $  \overline{3} \overline{2} \overline{1} $ \\
\hline
\end{tabular}
\end{center}

\begin{remark}
Colored Lehmer codes correspond to the interpretation given by Arslan \cite{Ars1} as the number of ways to place $k$ balls into $n$ boxes, where the $i$-th box can contain at most $2i-1$ balls, for type  B.
\end{remark}

\section{Properties of the Colored Mahonian Numbers}

In this section, we present several properties of the colored Mahonian numbers.  Using the framework of Lehmer codes, we establish a symmetry identity and a recurrence relation that involves integer partitions.  We also employ generating functions to generalize the Knuth–Netto formula. The section concludes with recurrence relations for the total number of colored inversions.

\begin{theorem}[Symmetry]
    For all $n>1$,  $c\geq 1$, and $0\leq k \leq (c-1) n + c\binom{n}{2}$, we have
       $$i_c(n,k)=i_c\left(n, (c-1) n + c\binom{n}{2}-k\right).$$
\end{theorem}
\begin{proof}
     Let $\ell \in \mathcal{L}^{(c)}_{n,k}$ be a colored Lehmer code. Since $0\leq \ell _i < c\cdot i$, define a new code $\tilde{\ell} = (\tilde{\ell}_1, \dots, \tilde{\ell}_n)$ by setting   $\tilde \ell _i = c\cdot i -1-\ell _i$ for each $i$. Then $\tilde{\ell}$ is also a valid colored Lehmer code. Adding over the components of $\tilde{\ell}$, we have 
    \begin{align*}
        \sum _{i = 1}^n \tilde \ell _i &= \sum _{i = 1}^n \left ( c\cdot i -1-\ell _i \right )=  c\binom{n+1}{2}-n-k \\&= c\left (n+\binom{n}{2}\right )-n-k = (c-1)n+c\binom{n}{2}-k.
    \end{align*}
    The correspondence $\ell \mapsto \tilde \ell$ is an involution and defines a bijection between $\mathcal{L}^{(c)}_{n,k}$ and $\mathcal{L}^{(c)}_{n,(c-1)n+c\binom{n}{2}-k}$. The result follows.
\end{proof}

\begin{theorem}[Summation property]
 For all $n>1$,  $c\geq 1$, and $0\leq k \leq (c-1) n + c\binom{n}{2}$, we have
 $$i_c(n,k)=\sum_{j=0}^{cn-1} i_c(n-1,k-j).$$
\end{theorem}
\begin{proof}
Using Theorem~\ref{thmcollehcode}, we provide a combinatorial proof based on colored Lehmer codes. Given a colored Lehmer code $\ell = (\ell _1,\dots ,
\ell _{n-1}) \in \mathcal{L}^{(c)}_{n-1,m}$, we can construct a new code $\hat \ell = (\ell_1,\dots,\ell _{n-1},j)$ by just concatenating $j$ at the end. If $0\leq j\leq c\cdot n-1$, then  $\hat \ell \in \mathcal{L}^{(c)}_{n,m+j}$. Adding over all possible $0\leq j<c\cdot i$ and all possible colored Lehmer codes $\ell \in \mathcal{L}^{(c)}_{n-1,m}$, we obtain all elements of $\mathcal{L}^{(c)}_{n, k}$ for $k = m + j$. By changing the variable $m = k - j$, the result follows.  
\end{proof}

\begin{theorem}[Recurrence relation]
 For all $n>1$,  $c\geq 1$, and $0\leq k \leq (c-1) n + c\binom{n}{2}$, we have
$$i_c(n,k)=i_c(n,k-1) + i_c(n-1,k) - i_c(n-1,k-cn).$$
\end{theorem}
\begin{proof}
    From the summation property, we have
  \begin{align*}
 i_c(n,k)-i_c(n,k-1)&=\sum_{j=0}^{cn-1} i_c(n-1,k-j)-\sum_{j=0}^{cn-1} i_c(n-1,k-1-j)\\
 &=i_c(n-1,k)-i_c(n-1,k-cn),
  \end{align*}
  which proves the recurrence.
 \end{proof}

We now present a generalization of the Knuth–Netto formula \cite{KN1, KN2} for computing the colored Mahonian numbers.

\begin{theorem}[Knuth-Netto formula]
For $0\leq k \leq n$, we have
$$i_c(n,k)=\binom{n+k-1}{k} + \sum_{\ell\geq 1}(-1)^j\binom{n + k - cu_j - cj - 1}{k - c u_j - cj} + \sum_{\ell\geq 1}(-1)^j\binom{n + k - cu_j- 1}{k - c u_j },$$
where $u_j=j(3j-1)/2$ is the $j$-th generalized pentagonal number.
\end{theorem}
\begin{proof}
From the generating function \eqref{gfcolr}, we have 
\begin{align*}
\sum_{k=0}^ni_c(n,k)q^k&=\prod_{i=1}^n\frac{1-q^{ic}}{1-q}\\
    &=\frac{1}{(1-q)^n}\prod_{i=1}^n(1-q^{ic})\\
    &=\prod_{i=1}^n(1-q^{ic})\sum_{\ell\geq 1}\binom{n+\ell+1}{\ell}q^{\ell}.
\end{align*}
Now, using Euler’s pentagonal number theorem (cf.\ \cite{Andrews}),  
$$\prod_{i=1}^{\infty}(1-q^i)=\sum_{k=-\infty}^{\infty}(-1)^kq^{u_k}=\sum_{k=-\infty}^{\infty}(-1)^k\left(q^{\frac{k(3k-1)}{2}}+q^{\frac{k(3k+1)}{2}}\right).$$
Substituting  $q$ by $q^c$, we have 
\begin{align*}
\sum_{k=0}^ni_c(n,k)q^k&=\prod_{i=1}^n(1-q^{ic})\sum_{\ell\geq 1}\binom{n+\ell+1}{\ell}q^{\ell}\\
&=\sum_{k=-\infty}^{\infty}(-1)^k\left(q^{\frac{ck(3k-1)}{2}}+q^{\frac{ck(3k+1)}{2}}\right)\sum_{\ell\geq 1}\binom{n+\ell+1}{\ell}q^{\ell}.
\end{align*}
Comparing coefficients of $q^k$ on both sides gives the result.
\end{proof}

In Theorem~\ref{MahC}, we express the colored Mahonian numbers in terms of the classical Mahonian numbers and a partition function. Let $p_{\leq i, j}(n)$ denote the number of integer partitions of $n$ into parts of size at most $i$, where each part appears at most $j$ times.

\begin{theorem}\label{MahC}
For $0\leq k\leq (c-1)n+c\binom{n}{2}$, we have
\[
i_{c}(n,k)=\sum_{j=0}^{k}i(n,j)p_{\leq n, c-1}(k-j).
\]
\end{theorem}
\begin{proof}
The coefficient of $q^{b}$ in the expansion of $\prod_{i=1}^{n}\left[ c\right]
_{q^{i}}$ counts the number of partitions of $b$ into parts of size at most $n$, where each part appears at most $c-1$ times. Therefore, we obtain the generating function:
\begin{align*}
\sum_{k=0}^{(c-1) n + c\binom{n}{2}}i_c(n,k)q^k&=[n]_{q}!\prod_{i=1}^{n}
\left[ c\right] _{q^{i}}\\
&=\sum_{k\geq 0}i(n,k)q^k\sum_{k\geq 0}p_{\leq n, c-1}(k)q^k\\
&=\sum_{k\geq 0}\sum_{j=0}^ki(n,k)p_{\leq n, c-1}(k-j)q^k.
\end{align*}
Equating coefficients of $q^k$ on both sides yields the desired identity.

We also provide a second proof using colored Lehmer codes. Let $\ell=(\ell_1, \dots, \ell_n) \in \mathcal{L}^{(c)}_{n,k}$, where $0\leq \ell _i<c\cdot i$. By the division algorithm, there exist integers $q_i$ and $r_i$ such that $$\ell _i= q_i\cdot i+r_i, \quad 0\leq r_i<i, \quad 1\leq i \leq n.$$ 
Define the vectors $\bm{q}_\ell = (q_1,\dots, q_n)$ and $\bm{r}_\ell= (r_1,\dots, r_n)$. Observe that $0\leq q_i<c$ and $0\leq r_i<i$, which implies that $\bm{r}_\ell \in \mathcal{L}_{n,r_1+\cdots +r_n}$. Moreover, the sum  $1\cdot q_1+2\cdot q_2+\cdots +n\cdot q_n = s$ encodes the partition of  $s$ into parts of size at most $n$, with each part appearing at most $c-1$ times, due to the constraint $q_i<c$.
\end{proof}

\begin{remark}
From the partition-based interpretation, we recover the combinatorial meaning of colored Mahonian numbers in terms of tilings and lattice paths, as described by  Ghemit and Ahmia \cite{GH}. The colored Mahonian number $i_c(n,k)$ counts the number of tilings of a $(n+k) \times 1$ board using exactly $n$ green and $k$ orange squares, such that after placing $j$ green squares, there are at most $cj - 1$ consecutive orange squares. Equivalently, it counts the number of lattice paths from  \( (0,0) \) to \( (n,k) \) using north \((0,1)\) and east \((1,0)\) steps, with at most $cj-1$ north steps at level $j \geq 1$.
\end{remark}

\subsection{The number of colored inversion}

Remember that $I_{c,n}$ denotes the total number of colored inversion in $G_{c,n}$. Let $C_{c,j}=\left\{ \sigma \in G_{c,n}: \sigma(j)=n^{[r]}\right\}$.  It is clear that $|C_{c,j}|=c^{n-1}(n-1)!$. We can also
express $G_{c,n}$ as the disjoint union
\begin{equation}\label{cje}
G_{c,n}=\biguplus _{j=1}^nC_{c,j}.
\end{equation}

Let \[
\sigma = \left( 
\begin{array}{ccccc}
1 & \cdots & j & \cdots & n \\ 
\sigma(1) & \cdots & n^{[r]} & \cdots & \sigma(n)
\end{array}
\right) \in C_{c,j}.
\]
From the colored permutation $\sigma$, we define the colored permutation  $\tau^{(j)}$ by
\[
\tau^{(j)} =\left( 
\begin{array}{cccccc}
1 & \cdots  & j-1 & j &\cdots&n-1 \\ 
\sigma (1) & \cdots  & \sigma(j-1) & \sigma(j+1) & \cdots & \sigma(n) 
\end{array}%
\right) \in G_{c,n-1}.
\]%
From \eqref{maindef} and Lemma 4.7 of \cite{Ars2}, we conclude that 
\begin{equation}\label{rc}
\inv_c(\sigma)=\begin{cases}
\inv_c(\tau^{(j)}) + n - j , &\text{if} \ r=0, \\
\inv_c(\tau^{(j)})+n-j+c(j-1)+r, &\text{if}  \ r\neq 0.
\end{cases}
\end{equation}

It is clear that 
\begin{equation*}
I_{c,n}=\sum_{\sigma\in G_{c,n}}\inv_c(\sigma)=\sum_{j=1}^n\sum_{\sigma \in C_{c,j}}\inv_c(\sigma).
\end{equation*}

From~\eqref{rc}, we can derive the following recurrence relation for the sequence $I_{c,n}$. Additionally, a combinatorial argument based on colored Lehmer codes provides further insight.

\begin{theorem} For all $n\geq 2$, we have
\[
I_{c,n}=c^{n}n!\frac{(cn-1)}{2}+cnI_{c,n-1},
\]
with the initial condition $I_{c,1}=\binom{c}{2}$. 
\end{theorem}
\begin{proof}
It is clear that $G_{c,1}=\{1,\overline{1},\dots,1^{[c-1]}\}$, so we obtain   $I_{c,1}=0+1+\cdots+c-1=\binom{c}{2}$. Now, suppose $n\geq 2$. For each $j\in \Sigma_{c,n}$, we have
\begin{align*}
I_{c,n}&=\sum_{j=1}^n\sum_{\sigma \in C_{c,j}}\inv_c(\sigma)\\&= \sum_{j=1}^n\sum_{\substack{\tau^{(j)} \in G_{c,n-1}\\ \col(\sigma_j)=0}}\left(\inv_c(\tau^{(j)})+n-j\right) + \sum_{j=1}^n\sum_{\substack{\tau^{(j)} \in G_{c,n-1}\\ \col(\sigma_j)=r\neq0}}\left(\inv_c(\tau^{(j)})+n-j+c(j-1)+r\right)\\
&= c^{n-1}(n-1)!\sum_{j=1}^n(n-j)+\sum_{j=1}^n\sum_{\substack{\tau^{(j)} \in G_{c,n-1}\\ \col(\sigma_j)=0}}\inv_c(\tau^{(j)}) \\
& +c^{n-1}(n-1)!\sum_{j=1}^n\sum_{r=1}^{c-1}(n-j + c(j-1)+r) + \sum_{j=1}^n\sum_{\substack{\tau^{(j)} \in G_{c,n-1}\\ \col(\sigma_j)=r\neq0}}\inv_c(\tau^{(j)}).
\end{align*}
Simplifying the above sums, we obtain  
\begin{align*}
I_{c,n}&= c^{n-1}(n-1)!\left(\frac{n(n-1)}{2} + \frac{(c-1)n(n + c n - 1)}{2}\right) + \sum_{j=1}^n\sum_{r=0}^{c-1}\sum_{\substack{\tau^{(j)} \in G_{c,n-1}\\\col(\sigma_j)=r}}\inv_c(\tau^{(j)}) \\
&=c^nn!\frac{(cn-1)}{2} +  \sum_{j=1}^n\sum_{r=0}^{c-1}\sum_{\sigma \in G_{c,n-1}}\inv_c(\sigma) \\
&=c^nn!\frac{(cn-1)}{2} + cnI_{c,n-1}.  &\qedhere
\end{align*}
\end{proof}
\begin{proof}[Combinatorial proof.]
Theorem~\ref{thmcollehcode} implies that $$ I_{c,n} = \sum _{\ell \in \mathcal{L}^{(c)}_{n}}\sum _{i = 1}^n \ell _i.$$ 
Since $0\leq \ell_n<cn$, we can rewrite this sum as follows:

\begin{align*}
    I_{c,n} &= \sum _{\ell \in \mathcal{L}^{(c)}_{n}}\sum _{i = 1}^n \ell _i= \sum _{\ell \in \mathcal{L}^{(c)}_{n-1}}\sum _{j = 0}^{cn-1}((\ell _1+\ell _2+\cdots +\ell _{n-1})+j)\\
    &= \sum _{\ell \in \mathcal{L}^{(c)}_{n-1}}cn(\ell _1+\ell _2+\cdots +\ell _{n-1})+\sum _{j = 0}^{cn-1}j= cnI_{c,n-1}+\sum _{\ell \in \mathcal{L}^{(c)}_{n-1}}\binom{cn}{2}\\
    &=cnI_{c,n-1}+c^{n-1}(n-1)!\frac{cn (cn-1)}{2} = cnI_{c,n-1}+c^{n}n!\frac{(cn-1)}{2}. & \qedhere
\end{align*}
\end{proof}

\begin{theorem}\label{eqIcn} For all $c\geq 1$ and $n\geq 0$, we have
    $$I_{c,n}=\frac{c^nn!}{2}\left (c\binom{n+1}{2}-n\right ).$$
\end{theorem}
\begin{proof}
Using Theorem~\ref{thmcollehcode}, we can sum over each component of the colored Lehmer codes as follows:
\begin{align*}
    I_{c,n} &= \sum _{i = 1}^n \sum _{\ell \in \mathcal{L}_{n}^{(c)}}\ell _i= \sum _{k = 1}^n \frac{c^nn!}{ck}\sum _{i = 0}^{ck-1}i\\
    &= \sum _{k = 1}^n \frac{c^nn!}{ck}\binom{ck}{2} = \frac{c^nn!}{2}\left (c\binom{n+1}{2}-n\right).
\end{align*}
The second equality follows from the  condition on the components of a colored Lehmer code, namely, that   $0\leq \ell _i < c\cdot i$.
\end{proof}

\begin{proof}[Second proof]
 Equation \eqref{qidentity1} can also be used to derive an explicit expression  for  $I_{c,n}$ by differentiating the left-hand side and setting $q=1$. Applying the generalized Leibniz rule, the right-hand side simplifies to
    $$\sum _{k = 1}^n\frac{[c]_q\cdots [cn]_q}{[ck]_q}\partial _q [ck]_q = \sum _{k = 1}^n\frac{[c]_q\cdots [cn]_q}{[ck]_q}\left (\sum _{i = 1}^{ck-1}i\cdot q^{i-1} \right ).$$
Setting $q=1$, we obtain
    $$I_{c,n} = \sum _{k = 1}^n\frac{c^n\cdot n!}{ck}\binom{ck}{2} = \frac{c^nn!}{2}\left (c\binom{n+1}{2}-n\right ).$$
\end{proof}

\begin{corollary} For all $n\geq 2$ and $c\geq 1$
    $$I_{c,n}=\frac{c n^2 (c n+c-2)}{(n-1) (c n-2)}I_{c,n-1}.$$
\end{corollary}

\section{Inversion in Colored Derangements}

An element $\sigma_1^{\col(\sigma_1)}\sigma_2^{\col(\sigma_2)}\cdots \sigma_n^{\col(\sigma_n)}$  of $G_{c,n}$ is called a  \emph{derangement}  if $\sigma_i\neq i$ for all  $i\in \Sigma$, meaning $\sigma$ has no fixed points.  Let  $\D_n^{(c)}$ denote the set of all derangements of  $G_{c,n}$.
For example,  the colored  permutation $5\,\overline{6}\,\overline{\overline{3}}\,2 \, \overline{4} \,  \overline{1}$ is a derangement in $\D_{6}^{(3)}$. 

Faliharimalala and Zeng \cite{FZ} (see also \cite{Assaf}) established the explicit formula
\begin{align*}
d_{n}^{(c)}:=|\D_n^{(c)}|=n!\sum_{k=0}^n\frac{(-1)^kc^{n-k}}{k!},
\end{align*}
as well as the recurrence relation  
\begin{align*}
d_{n+1}^{(c)}=(cn+c)d_n^{(c)}+(-1)^{n+1}, \quad n\geq 0.
\end{align*}
We refer to $d_{n}^{(c)}$  as the  \emph{colored derangement numbers}. Further properties of these numbers were studied by  Mez\H{o} and Ram\'irez \cite{MR};  see also a further generalization in \cite{Moll}.

Let $t_{n}^{(c)}$ denote the total number of inversions in all colored derangements of size $n$. In  the classical case $c=1$, Alekseyev \cite{Alek} (see also sequence \seqnum{A216239} in \cite{OEIS}) provided  explicit expressions for this sequence. To generalize this result to the colored setting, we will present a proof based on the Principle of Inclusion-Exclusion (PIE).

\begin{theorem}\label{theo41} For $n\geq 0$, the total number of inversions in all classical derangements of size $n$ is given by
    $$t_{n}^{(1)}=\frac{n!}{12}\sum_{k=0}^{n-1}(-1)^k\frac{(3n+k)(n-k-1)}{k!}.$$
%Moreover, 
%$$t_{n}^{(1)}=\frac{1}{12}\left((3n^2 - n + 1)d_n^{(1)}+(n - 1)(-1)^n\right).$$
\end{theorem}
\begin{proof}
We apply the Principle of Inclusion-Exclusion similarly to the classical case of derangements. Suppose we fix  exactly \( k \) fixed points, say \( p_1, p_2, \dots, p_k \) with \( p_1 < p_2 < \cdots < p_k \). These can be chosen in \( \binom{n}{k} \) ways.

Now, let us count the total number of inversions before fixing these $k$ elements, that is the inversions independent to the $p_i$'s. Since there are $n-k$ free positions remaining, we apply the formula from  Theorem~\ref{eqIcn} with $c=1$ to determine the number of inversions over all permutations of size $n-k$.  This yields
\[
\frac{(n-k)!}{2} \left(\binom{n-k+1}{2} - (n-k) \right) = \frac{(n-k)!(n-k)(n-k-1)}{4}.
\]
Each choice of $k$ fixed positions contributes the same number of inversions, and there are $\binom{n}{k}$ such choices. Thus, the total number of inversions associated with all permutations fixing exactly \( k \) positions is given by
$$a(n,k):=\binom{n}{k}\frac{(n-k)!(n-k)(n-k-1)}{4}=\frac{(n-k) (n- k-1) n!}{4k!}.$$ 

%Supongamos que utilizamos PIE, como en el caso clasico de deranmgements, si fijamos que hay $k$ puntos fijos, digamos $p_1, p_2, \dots, p_k$, con $p_1<p_2<\cdots <p_k$. Estos se pueden elegir de $\binom{n}{k}$ formas. Ahora contemos las inversiones, antes de fijar esos $k$ elementos tenemos $n-k$ posiciones libres y utilizando la fórmula del Theorem \ref{eqIcn} para $c=1$, hay en total 
%$$\frac{(n-k)!}{2}\left(\binom{n-k+1}{2}-(n-k)\right)=\frac{(n-k)!(n-k)(n-k-1)}{4}.$$ inversiones. 
%Así que tenemos en total
%$$A(n,k):=\binom{n}{k}\frac{(n-k)!(n-k)(n-k-1)}{4}=\frac{(n-k) (n- k-1) n!}{4k!}.$$ 

Next, we need to count the inversions relative to the fixed elements. Consider a fixed point \( p_\ell \). The number of positions before \( p_\ell \) that are not fixed points is  \( p_\ell - \ell \). Additionally, there are \( n - k - (p_\ell - \ell) \) elements that are not fixed points but are greater than \( p_\ell \).  Thus, any of these larger elements appearing in the first \( p_\ell - \ell \) free positions contribute to the inversion count, leading to $(p_\ell - \ell)(n - k - (p_\ell - \ell))$ new inversions. Similarly, after \( p_\ell \), there are \( (n - p_\ell) - (k - \ell) \) positions that are not fixed points, and there are \( p_\ell - \ell \) elements that are not fixed points and are smaller than \( p_\ell \). This results in the following additional inversions $((n - p_\ell) - (k - \ell))(p_\ell - \ell)$.

Finally, the remaining elements that are not fixed points can be arranged in \( (n-k-1)! \) different ways. Summarizing, we obtain:
\begin{multline*}
    \left ((p_\ell - \ell)(n - k - (p_\ell - \ell)) + ((n - p_\ell) - (k - \ell))(p_\ell - \ell)\right ) (n-k-1)! \\= 2(n-k-1)! (p_\ell - \ell)(n-k-p_\ell+\ell).
\end{multline*}

By the Addition Principle, we obtain:
 \begin{align*}
b(n,k):=&2(n-k-1)!\sum_{1\leq p_1<p_2<\cdots< p_k}\sum_{\ell=1}^k(p_\ell-\ell)(n-k-p_\ell+\ell)\\
=&2(n-k-1)!\sum_{\ell = 1}^k \sum _{p_{\ell}=\ell}^{n-(k-\ell)}\binom{p_{\ell}-1}{\ell -1}\binom{n-p_{\ell}}{k-\ell}(p_{\ell}-\ell)(n-k-p_{\ell}+\ell)\\
=&2(n-k-1)!\sum_{\ell = 1}^k \sum _{p_{\ell}=\ell}^{n-(k-\ell)}\binom{p_{\ell}-1}{p_{\ell}-\ell}\binom{n-p_{\ell}}{n-k-(p_{\ell}-\ell)}(p_{\ell}-\ell)(n-k-p_{\ell}+\ell)\\
=&2(n-k-1)!\sum_{j = 0}^{n-k} \sum _{p_{\ell}=1}^{n}\binom{p_{\ell}-1}{j}\binom{n-p_{\ell}}{n-k-j}j(n-k-j)\\
=&2(n-k-1)!\sum_{j = 0}^{n-k}j(n-k-j) \sum _{p_{\ell}=1}^{n}\binom{p_{\ell}-1}{j}\binom{n-p_{\ell}}{n-k-j}\\
=&2(n-k-1)!\sum_{j = 0}^{n-k}j(n-k-j) \binom{n}{k-1}.
\end{align*}

In the last equality we use the identity (5.26) of \cite{CM}.  Simplifying the last sum, we obtain 
$$b(n,k)=\frac{(n-k-1)n!}{3(k-1)!}.$$
By the PIE, we have
\begin{align*}
 t_n^{(1)}=\frac{1}{4} (n-1) n n! -   \sum_{k=1}^{n-1}(-1)^{k-1}(a(n,k) + b(n,k)).
\end{align*}
Substituting the expression for $a(n,k)$ and $b(n,k)$ we obtain the desired result. 
%\begin{align*}
 % t_n^{(1)}&= \sum_{k=0}^{n-1}(-1)^k\left(\frac{(n-k) (n- k-1) n!}{4k!} + \frac{(n-k-1)n!}{3(k-1)!}\right) 
  %&=\sum_{k=0}^{n-1}(-1)^k\left(\frac{(n-k) (n- k-1) n!}{4k!} + \frac{(n-k-1)kn!}{3 k!}\right)\\
     %&=\sum_{k=1}^{n-1}(-1)^k\left(\frac{3(n-k) (n-k-1) n!+ 4(n-k-1)kn!}{12 k!}\right)\\
   %&=\frac{n!}{12}\sum_{k=0}^{n-1}(-1)^k\left(\frac{3(n-k) (n-k-1) + 4(n-k-1)k}{k!}\right)\\
            %    &=\frac{n!}{12}\sum_{k=0}^{n-1}(-1)^k\left(\frac{(n-k-1)(3(n-k)+4k)}{k!}\right)\\
                 %=\frac{n!}{12}\sum_{k=0}^{n-1}(-1)^k\left(\frac{(n-k-1)(3n+k)}{k!}\right).
%\end{align*}
\end{proof}

In the following theorem, we generalize the previous proof to the colored case and provide an explicit formula for the sequence $t_n^{(c)}$, which counts the total number of inversions in all colored derangements of size $n$.

\begin{theorem}
For all  $n \geq 0$ and $c \geq 1$, we have 
    \begin{align*}
  t_n^{(c)}= & \frac{n!}{12}\sum _{k = 0}^{n-1}\left (-1\right )^k\frac{c^{n-k}(n-k-1)(3n+k)}{k!}\\
    &+ n!\sum _{k = 0}^n\frac{(-1)^k}{k!}\sum _{i = 0}^{(n-k)(c-1)}i\sum _{j = 0}^{n-k}(-1)^j\binom{n-k}{j}\binom{i-cj+n-k-1}{n-k-1}\\
    &+\frac{n!(c-1)}{2}\sum _{k = 0}^n\frac{(-1)^kc^{n-k}}{k!}\binom{n-k}{2}\\
    &+\frac{n!(c-1)}{6}\sum _{k = 1}^{n-1}\frac{(-1)^kc^{n-k}(2(n-k)+1)}{(k-1)!}.
\end{align*}
\end{theorem}

\begin{proof}
 A colored derangement is a colored permutation such that, if the underlying permutation has a fixed point, then it must have a nonzero color. This observation implies that the same inclusion-exclusion argument used in the proof of Theorem~\ref{theo41} works under the additional assumption that the color of each fixed point $p_i$ is $0$.
 
 From \eqref{maindef}, we can decompose the sequence $t_{n}^{(c)}$ as $t_{n}^{(c)}= A+B+C$, where $A$ deals with  $\inv(|\sigma|)$, $B$ corresponds to the contribution from the coloring statistic $\col(\sigma)$, and $C$ captures the contribution from \[
 c\sum _{\substack{i<j, \, \sigma _i <\sigma _j\\c_j\neq 0}}1.
\]  
  
For the summand $A$, we follow a similar approach as in Theorem~\ref{theo41}. The key observation is that the inversion count of the underlying permutation is independent of the coloring. Hence, for each choice of $k$ fixed points, the remaining $n-k$ non-fixed elements can be colored in $c^{n-k}$ ways.
Therefore, we obtain 
$$A = \frac{n!}{12}\sum _{k = 0}^{n-2}\left (-1\right )^k\frac{c^{n-k}(n-k-1)(3n+k)}{k!}.$$

 For the summand $B$, we observe that the underlying permutation is independent of the coloring, so we only need to keep track of the number of elements assigned to each color. To do this, we consider integer compositions. Let $\mathcal{C}_{n,k,<c}$ denote the set of integer compositions of $n$ into $k$ non-negative parts, where each part is less than $c$. By  the inclusion-exclusion formula and the stars and bars method, it is known that 
  $$|\mathcal{C}_{n,k,<c}|=\sum _{j = 0}^{k}(-1)^j\binom{k}{j}\binom{n-cj+k-1}{k-1}.$$
We now compute $B$ as follows:
\begin{align*}
    B &= \sum _{k = 0}^n(-1)^k\sum _{p_1<\cdots <p_k}\sum _{\substack{\sigma\in G_{c,n}\\\sigma (p_i)=p_i, \,i\in [n]}}\texttt{col}(\sigma)\\
    &=\sum _{k = 0}^n(-1)^k(n-k)!\binom{n}{k}\sum _{i = 0}^{(n-k)(c-1)}i\left | \mathcal{C}_{i,n-k,<c}\right |\\
    &=\sum _{k = 0}^n(-1)^k(n-k)!\binom{n}{k}\sum _{i = 0}^{(n-k)(c-1)}i \sum _{j = 0}^{n-k}(-1)^j\binom{n-k}{j}\binom{i-cj+n-k-1}{n-k-1}\\
    &=n!\sum _{k = 0}^n\frac{(-1)^k}{k!}\sum _{i = 0}^{(n-k)(c-1)}i \sum _{j = 0}^{n-k}(-1)^j\binom{n-k}{j}\binom{i-cj+n-k-1}{n-k-1},
\end{align*}
where the second sum corresponds to the choice of fixed points $p_1, \dots, p_k$ in $\binom{n}{k}$ ways, the remaining $n-k$ elements are permuted in $(n-k)!$ ways, and then adding over integer compositions for the remaining $n-k$ positions that represents the sum of the colors parametrized by $i$.

Finally, for the third summand $C$, we have
$$C = \sum _{k = 0}^n(-1)^k\sum _{p_1<\cdots <p_k}\sum _{\substack{\sigma\in G_{c,n}\\\sigma (p_i)=p_i,\, i\in [n]}}c\sum _{\substack{i<j, \sigma _i<\sigma_j\\c_j\neq 0}}1.$$
We begin by computing the sum over all colored permutations $\sigma$ of the expression  
$$c\sum _{\substack{i<j, \sigma _i<\sigma_j\\c_j\neq 0}}1.$$
For this reason, consider $$C_{G_{c,n}} = \sum _{\sigma \in G_{c,n}}c\sum _{\substack{i<j, \sigma _i<\sigma_j\\c_j\neq 0}}1.$$
To evaluate this, we choose positions
$i<j$ in $\binom{n}{2}$ ways, permute the remaining $n-2$ elements in $(n-2)!$ ways, and assign colors in  $c^{n-2}$ ways to the positions other than $i$ and $j$. We then choose a color for  $i$ (with $c$ options) and one color for $j$ (with $c-1$ options). This yields:
\begin{align*}
    C_{G_{c,n}}  &= c\sum _{i<j}c^{n-2}\cdot c\cdot(c-1)\binom{n}{2}(n-2)! \\
    &= \frac{c^nn!(c-1)}{2}\binom{n}{2}.
\end{align*}

Next, we distinguish two cases for the indices $i$ and $j$: either (1) $i$ and $j$ are not among the fixed points $\{p_{\ell}\}_{\ell \in [k]}$, or  (2) $i$ is a fixed point while  $j$ is not.

For the first case, as the indices are not part of the fixed points, it is equivalent to considering all possible permutations of the remaining elements in $G_{c,n-k}$. Thus, we obtain
\begin{align*}
C_1 &= \sum _{k = 0}^n(-1)^k\binom{n}{k}C_{G_{c,n-k}}\\
&= \sum _{k = 0}^n(-1)^k\binom{n}{k}\frac{c^{n-k}(n-k)!(c-1)}{2}\binom{n-k}{2}\\ 
&= \frac{n!(c-1)}{2}\sum _{k = 0}^n\frac{(-1)^kc^{n-k}}{k!}\binom{n-k}{2}.
\end{align*}
For the second case, we take $i$ to be one of the $p_{\ell}$'s, say $p_r$, and $j$ to be one of the elements not among the $p_{\ell}$'s. We denote this sum by $C_2$, and  we obtain
\begin{align*}
    C_2
    &=\sum _{k = 1}^n(-1)^k\sum _{p_1<\cdots <p_k}\sum _{\substack{\sigma\in G_{c,n}\\\sigma (p_i)=p_i,\, i\in [n]}}c\sum _{\substack{i<j, i = p_r\\ p_r<\sigma _j\\c_j\neq 0}}1\\
    &=\sum _{k = 1}^n(-1)^k\sum _{p_1<\cdots <p_k}\sum _{\ell = 1}^k(n-p_{\ell}-(k-\ell))^2c(c-1)(n-k-1)!c^{n-k-1}\\
    &=(c-1)\sum _{k = 1}^n(-1)^k(n-k-1)!c^{n-k}\sum _{p_1<\cdots <p_k}\sum _{\ell = 1}^k(n-p_{\ell}-(k-\ell))^2\\
    &=(c-1)\sum _{k = 1}^n(-1)^k(n-k-1)!c^{n-k}\sum _{\ell = 1}^k\sum _{p_{\ell} = \ell}^{n-(k-\ell)}\binom{p_{\ell}-1}{\ell -1}\binom{n-p_{\ell}}{k-\ell}(n-p_{\ell}-(k-\ell))^2\\
    &=(c-1)\sum _{k = 1}^n(-1)^k(n-k-1)!c^{n-k}  \sum _{\ell = 1}^k\sum _{p_{\ell} = \ell}^{n-(k-\ell)}\binom{p_{\ell}-1}{p_{\ell}-\ell }\binom{n-p_{\ell}}{n-k-(p_{\ell}-\ell)}(n-k-(p_{\ell}-\ell))^2.
\end{align*}

 Now, introducing the change of variable $j = p_{\ell} - \ell$, we rewrite the last equality as  
\begin{align*}
    C_2
    &=(c-1)\sum _{k = 1}^{n-1}(-1)^k(n-k-1)!c^{n-k}\sum _{j = 0}^{n-k}\sum _{p_{\ell} = 1}^{n}\binom{p_{\ell}-1}{j }\binom{n-p_{\ell}}{n-k-j}(n-k-j)^2\\
    &=(c-1)\sum _{k = 1}^{n-1}(-1)^k(n-k-1)!c^{n-k}\sum _{j = 0}^{n-k}(n-k-j)^2\sum _{p_{\ell} = 1}^{n}\binom{p_{\ell}-1}{j }\binom{n-p_{\ell}}{n-k-j}\\
    &=(c-1)\sum _{k = 1}^{n-1}(-1)^k(n-k-1)!c^{n-k}\sum _{j = 0}^{n-k}(n-k-j)^2\binom{n}{n-k+1}\\
    &=(c-1)\sum _{k = 1}^{n-1}(-1)^k(n-k-1)!c^{n-k}\binom{n}{k-1}\sum _{j = 0}^{n-k}j^2\\
    &=(c-1)\sum _{k = 1}^{n-1}(-1)^k(n-k-1)!c^{n-k}\frac{n!(n-k)(n-k+1)(2(n-k)+1)}{6(k-1)!(n-k+1)!}\\
 %   &=(c-1)\sum _{k = 1}^{n-1}(-1)^kc^{n-k}\frac{n!(2(n-k)+1)}{6(k-1)!} \\
    &=\frac{n!(c-1)}{6}\sum _{k = 1}^{n-1}\frac{(-1)^kc^{n-k}(2(n-k)+1)}{(k-1)!} 
\end{align*}
Putting it all together, $t_{n}^{(c)}= A+B+\left (C_1+C_2\right )$, we obtain the desired result.
\end{proof}

In Table \ref{tab2}, we show the first few values of the sequence $t_{n}^{(c)}$.

\begin{table}[htp]
\begin{tabular}{l||lllllll}
  $c\backslash n$ &  $1$ & $2$ & $3$ & $4$ & $5$ & $6$ & $7$ \\ \hline\hline
   $c= 1 $&0 &1 &4 &34 &260 &2275 &21784 \\
$c= 2 $&1 &12 &149 &2048 &31345 &534524 &10091893 \\
$c= 3 $&3 &51 &945 &19230 &438195 &11146977 &314382369 \\
$c= 4 $&6 &136 &3334 &90016 &2725430 &92206024 &3460668694 \\
$c= 5 $&10 &285 &8690 &292490 &11047750 &466523015 &21862413790 \\
$c= 6 $&15 &516 &18819 &758784 &34348335 &1738896660 &97715487219 \\
$c= 7 $&21 &847 &35959 &1689478 &89145245 &5261671429 &344776287583 \\
$c= 8 $&28 &1296 &62780 &3368000 &202965820 &13684414352 &1024393015324 \\
$c= 9 $&36 &1881 &102384 &6175026 &418429080 &31725812619 &2671026822324 
\end{tabular}
\caption{Values of $t_{n}^{(c)}$ for $1\leq n \leq 7$ and $1\leq c\leq 9$.}\label{tab2}
\end{table}

\begin{openproblem}
Let $t_c(n,k)$ denote the number of colored derangements of size $n$ with exactly $k$ inversions. 
Find an expression to calculate this sequence. Additionally, determine the generating functions for the sequences \( t_c(n,k) \) and \( t^{(c)}(n) \).\end{openproblem}

\section{Inversion in Colored Involutions}

In this final section, we study the inversion statistic on the set of colored involutions. As in the classical setting, colored permutations can be expressed as a product of disjoint cycles. For example, the colored permutation 
$\overline{\overline{5}}\,2\,\overline{1}\,\overline{\overline{4}}\,\overline{3}\in G_{3,5}$ can be written in cycle notation as \[
(1\,\overline{\overline{5}}\,3\,\overline{1}\,5\,\overline{3}\,\overline{\overline{1}}\,\overline{5}\,\overline{\overline{3}})(2)(\overline{2})(\overline{\overline{2}})(4\,\overline{\overline{4}}\,\overline{4}).
\]
A \emph{colored involution} is a colored permutation $\sigma \in G_{c,n}$ satisfying $\sigma^2 = \texttt{id}$, where $\texttt{id}$ denotes the identity permutation. For example, the colored permutation $21\overline{\overline{4}}\,\overline{3}\,5 \in G_{3,5}$ is a colored involution. It can be written in cycle notation as \[
(1\,2)(\overline{1}\,\overline{2})(\overline{\overline{1}}\,\overline{\overline{2}})(3\,\overline{\overline{4}})(\overline{3}\,4)(\overline{\overline{3}}\,\overline{5})(5)(\overline{4})(\overline{\overline{5}}).
\]
Observe that a colored permutation is a colored involution if and only if each cycle in its decomposition has length at most two. 

From the definition, we consider the following cases for $a\in [n]$: 
\begin{enumerate}
    \item  If $\sigma (a) = a^{[c_1]}$, then 
    \[\sigma ^2 (a) = \sigma(\sigma (a)) = \sigma (a^{[c_1]}) = \sigma (a)^{[c_1]} = a^{[2c_1]}.\] For $\sigma$ to be an involution, we must have $a^{[2c_1]} = a^{[0]}$, which implies $2c_1 \equiv 0 \pmod{c}$. That is, either $c_1 = 0$ or $c_1 = \frac{c}{2}$ if $c$ is even.

    \item If $\sigma(a) = b^{[c_1]}$ and $\sigma(b) = a^{[c_2]}$, then
    \[
    \sigma^2(a) = \sigma(\sigma(a)) = \sigma(b^{[c_1]}) = \sigma(b)^{[c_1]} = a^{[c_1 + c_2]}.
    \]
   In order for $\sigma$ to be an involution, we must have $c_1 + c_2 \equiv 0 \pmod{c}$.
\end{enumerate}

Let $r_{n}^{(c)}$ denote the total number of colored involutions in all colored permutations of size $n$. 

\begin{theorem}
For all $n\geq 0$ and $c\geq 1$, we have
    $$r_n^{(c)}= \sum _{\substack{0\leq k \leq n\\ n\equiv k \pmod 2}}\binom{n}{k}\left (3+(-1)^c\right )^k\frac{c^{\frac{n-k}{2}}(n-k)!}{2^{\frac{n+k}{2}}\left (\frac{n-k}{2}\right )!}.$$
    Moreover, the exponential generating function for the sequence $r_n^{(c)}$ is $$\exp \left (\frac{((-1)^c+3)x+cx^2}{2}\right ).$$
\end{theorem}
\begin{proof}
    Let $\sigma$ be a colored involution. Then $\sigma$ can be written as a product of fixed points and cycles of size two. 

    Suppose $\sigma$ has exactly $k$ fixed points. These can be chosen in $\binom{n}{k}$ ways. Each fixed point can have only one color if $c$ is odd (interpreted as having no color), or two possible colors if $c$ is even (either no color or color $c/2$). Hence, the number of ways to assign colors to the fixed points is $$\left (1+\left (\frac{(-1)^c+1}{2}\right )\right )^k.$$ 
    
    The remaining $n-k$ elements must be arranged in cycles of size two,  so $n-k$ must be even. The number of ways to form these cycles is \[\frac{(n-k)!}{2^{\frac{n-k}{2}}\left (\frac{n-k}{2}\right )!},\] since we permute the elements, group them into non-ordered pairs, and ignore the order of the pairs.
    
    For the coloring of each cycle of size two, we notice that in order for the permutation to be an involution, the colors $r_a$ and $r_b$ of the elements in the cycle must satisfy $r_a+r_b = c$. There are exactly $c$ such pairs,  so there are $c^{\frac{n-k}{2}}$ ways to assign colors to all these elements. 
    
    Summing over all valid values of $k$, namely those such that $n - k$ is even, gives the stated formula.

The exponential generating function follows directly from the closed-form expression.
\end{proof}

In Table \ref{tab3} we show the first few values of the sequence $r_{n}^{(c)}$.
\begin{table}[htp]
\begin{tabular}{l||lllllll}
  $c\backslash n$ &  $1$ & $2$ & $3$ & $4$ & $5$ & $6$ & $7$ \\ \hline\hline
   $c= 1 $&1 & 1 & 2 & 4 & 10 & 26 & 76 \\
$c= 3 $&1 & 2 & 6 & 20 & 76 & 312 & 1384 \\
$c= 4 $& 1 & 1 & 4 & 10 & 46 & 166 & 856 \\
$c= 5 $&1 & 2 & 8 & 32 & 160 & 832 & 4864 \\
$c= 6 $&1 & 1 & 6 & 16 & 106 & 426 & 3076 \\
$c= 7 $&1 & 2 & 10 & 44 & 268 & 1592 & 11224 \\
$c= 8 $&1 & 2 & 12 & 56 & 400 & 2592 & 21184 \\
$c= 9 $&1 & 1 & 10 & 28 & 298 & 1306 & 14716 
\end{tabular}
\caption{Values of $r_{n}^{(c)}$ for $1\leq n \leq 7$ and $1\leq c\leq 9$.}\label{tab3}
\end{table}

Let $i_{n}^{(c)}$ denote the total number of inversions in all colored involutions of size $n$, and let $I_{c,n}$ denote the set of colored involutions in $G_{c,n}$. We first consider the classical case  $c = 1$. 

\begin{theorem}
\label{lemmaInvolutions}
For any integer $n \geq 0$, the total number of inversions in all involutions  of size $n$ is given by 
    $$i^{(1)}_{n} = \sum _{\sigma \in I_{1,n}}\inv (\sigma) = \binom{n}{2}r_{n-2}^{(1)}+2\binom{n}{3}r_{n-3}^{(1)} + 6\binom{n}{4}r_{n-4}^{(1)}.$$
    \end{theorem}

\begin{proof}
We begin by rewriting the total number of inversions as
$$i^{(1)}_{n} = \sum _{\sigma \in I_{1,n}}\inv (\sigma) = \sum _{\sigma \in I_{1,n}}\sum _{i<j\,,\,\sigma _j>\sigma _i}1 = \sum _{i<j}\left | \{\sigma \in I_{n,1}: i<j \text{ and  }\sigma _i>\sigma _j\}\right |.$$
    
    We now classify and count the possible ways in which $\sigma_i > \sigma_j$ can occur given that $i<j$, depending on the types of cycles involved in $\sigma$.
   
    \begin{enumerate}
        \item  $\sigma _i  = j$ and $\sigma _j  = i$. This case contributes one inversion. There are $\binom{n}{2}$ such pairs, and removing the pair $(i,j)$ yields an involution of size $n-2$,  so we get a total contribution of $\binom{n}{2}r_{n-2}^{(1)}$.
        
        \item $\sigma _i = i$ and $\sigma _j = k \neq j$. In this case, for $\sigma_i > \sigma_j$ to hold, we must have $\sigma_j = k < i < j$. Once the triple $(i,j,k)$ is removed, we are left with an involution of size $n-3$, so this configuration contributes $\binom{n}{3}r_{n-3}^{(1)}$.
        
        \item  $\sigma _i = k\neq i$ and $\sigma _j = j$. In this case the order must be $i<j<k=\sigma _i$. This case is symmetric to the previous one and contributes another $\binom{n}{3}r_{n-3}^{(1)}$.
        
        \item $\sigma _i = k$ and $\sigma _j = \ell$.  The inversion condition $\sigma_i > \sigma_j$ becomes $k > \ell$ with $i < j$. There are 6 such ways to choose and order $(i,j,k,\ell)$ among the 4 elements, and removing all four gives an involution of size $n-4$. So, this case contributes $6 \binom{n}{4} r_{n-4}^{(1)}$. 
            \end{enumerate}
Adding the contributions from all cases completes the proof.
\end{proof}

We now generalize the previous result to colored involutions using the statistic $\inv_c$. We compute $$i_n^{(c)} = \sum _{\sigma \in I_{c,n}}\inv _c(\sigma).$$

\begin{theorem}
For all $n\geq0$ and $c\geq 1$, we have
  \begin{multline*}
       i_n^{(c)}= \frac{nc\left ((-1)^c+1\right )}{4}r_{n-1}^{(c)} + c \left (c+1+(-1)^c\right ) \binom{n}{2}r_{n-2}^{(c)}\\
       +2c^2\left ((-1)^c+2\right )\binom{n}{3}r_{n-3}^{(c)}+6c^3\binom{n}{4}r_{n-4}^{(c)}.
   \end{multline*} 
    \end{theorem}

\begin{proof}
    Using Equation \eqref{maindef}, we decompose the problem into computing the corresponding three distinct contributions:
    \begin{enumerate}
        \item $A = \sum _{\sigma\in I_{c,n}} \inv (|\sigma|)$.  \\       The proof of Theorem~\ref{lemmaInvolutions} can be extended to this more general setting. The only difference is that we must now consider the number of valid colorings in each case.  

          In the first case, when $\sigma_i = j$ and $\sigma_j = i$, the elements $i$ and $j$ form a cycle of size two. Such a cycle can be colored in $c$ different ways, corresponding to the number of valid colorings that satisfy the involution condition.

            In the second case, we consider configurations where there is one fixed point and one cycle of size 2.  The fixed point can be colored in either 1 or 2 ways, depending on the parity of $c$. The cycle of size 2, as before, can be colored in $c$ ways.    

            In the third case, we consider two disjoint cycles of length 2. Each cycle can be colored in $c$ ways independently, giving a total of $c^2$ colorings for the pair.
            
        Combining all these contributions, we obtain the following expression for $A$:

          %For the case of one fixed point and a cycle, one has one or two colors for the fixed point, depending on the parity of $c$, and $c$ colors for the cycle of length two.For the case of two cycles of length two, each one can be colored in $c$ ways, giving $c^2$. All together yields the following expression for $A$
$$A = c\binom{n}{2}r_{n-2}^{(c)} + 2c\left(1+\frac{(-1)^c+1}{2}\right)\binom{n}{3}r_{n-3}^{(c)}+6c^2\binom{n}{4}r_{n-4}^{(c)}.$$

        \item $B =  \sum _{\sigma\in I_{c,n}} \col(\sigma)$. \\  We now compute the total sum of colors over all colored involutions. This can be done by grouping contributions according to the types of cycles: either fixed points  or cycles of size 2.
        
First, consider the fixed points. A fixed point contributes a color only when $c$ is even, in which case the possible colors are $0$ and $c/2$.  Thus, for each of the $n$ positions, we may select it to be a fixed point, assign it the color $c/2$, and count the number of colored involutions on the remaining $n-1$ elements. The total contribution from this case is therefore $(c/2)n r_{n-1}^{(c)}$,  valid only when $c$ is even.

Next, consider the cycles of size 2.  We know that the two colors will add up to $c$ and so we pick the cycle in $\binom{n}{2}$ ways and the color of the smallest element of the cycle in $c-1$ ways which yields $c(c-1)\binom{n}{2}r_{n-2}^{(c)}$ ways. 

Combining both contributions, we obtain:
        $$B = \frac{n\cdot c}{2}\frac{((-1)^c+1)}{2}r_{n-1}^{(c)}+c(c-1)\binom{n}{2}r_{n-2}^{(c)}.$$
       
        \item $C = \displaystyle \sum _{\sigma \in I_{c,n} } \sum _{\substack{i<j, \sigma _i<\sigma _j\\c_j\neq 0}}1=\sum _{i<j}|\{\sigma \in I_{c,n}: \sigma _i< \sigma _j \text{ and } c_j\neq 0\}|.$ \\
        We follow a similar case analysis as in the computation of $A$ or Theorem~\ref{lemmaInvolutions}, but note that here we cannot have  
        $\sigma _i = j$ and $\sigma _j = i$. Instead, we have the case where both $i$ and $j$ are fixed points. If $c$ is even, there are two options of colors for $i$ and if $c$ is odd, then $j$ has no color. This yields $2  \left (\frac{(-1)^c + 1}{2}\right ) \binom{n}{2}r_{n-2}^{(c)}$  as the total contribution of this case, accounting for the parity of $c$.
        
        For the case $\sigma _j = j$ and $\sigma _ i = k$, we can choose $i,j,k$ with the condition $i,k<j$ in $2\binom{n}{3}$ ways.  For $j$ to contribute a color, $c$ must be even, and $i,k$ form a cycle, so we color it with $c$ colors. For the case $\sigma _i = i$ and $\sigma _ j = k$, we can choose $i,j,k$ with the condition $i<k,j$ in $2\binom{n}{3}$ and there are $c-1$ options of color for the pair $k,j$ and condition to the parity of $i$ there are either $1$ or $2$. For the case in which $\sigma _i = k$ and $\sigma _j=\ell$, we can choose these in $\binom{n}{4}$ ways but the condition $i<j$ and $k<\ell$ says that we can choose $2$ out of these form to know who $i,j$ are, and so we can do this in $\binom{n}{4}\binom{4}{2} = 6\binom{n}{4}$ ways. There are $c$ ways to choose the color for the pair $i,k$ and $c-1$ for the pair $j,\ell$. Summing over all these cases we get
        \begin{multline*}
            C= ((-1)^c+1)\binom{n}{2}r_{n-2}^{(c)}+((-1)^c+3)\binom{n}{3}(c-1)r_{n-3}^{(c)}\\
            +((-1)^c+1)c\binom{n}{3}r_{n-3}^{(c)} + 6c(c-1)\binom{n}{4}r_{n-4}^{(c)}.
        \end{multline*}
    \end{enumerate}
Combining all contributions, we conclude that  $i^{(c)}_{n} = A+B+c\cdot C,$ which completes the proof.
\end{proof}

In Table \ref{tab4} we show the first few values of the sequence $i_{n}^{(c)}$.
\begin{table}[htp]
\begin{tabular}{l||llllllll}
  $c\backslash n$ &  $1$ & $2$ & $3$ & $4$ & $5$ & $6$ & $7$ & $8$\\ \hline\hline
   $c= 1 $&0 & 1 & 5 & 26 & 110 & 490 & 2086 & 9240 \\
$c= 2 $&1 & 12 & 90 & 608 & 3900 & 24912 & 159544 & 1036800 \\
$c= 3 $& 0 & 9 & 45 & 450 & 2430 & 19530 & 117054 & 904680 \\
$c= 4 $&2 & 32 & 288 & 2560 & 20800 & 175104 & 1455104 &
   12517376 \\
$c= 5 $&0 & 25 & 125 & 1850 & 10750 & 123250 & 829150 & 8911000
   \\
$c= 6 $&3 & 60 & 594 & 6432 & 59700 & 606096 & 5862360 &
   60569088 \\
$c= 7 $&0 & 49 & 245 & 4802 & 28910 & 429730 & 3065734 &
   42024360 \\
$c= 8 $&4 & 96 & 1008 & 12800 & 129600 & 1525248 & 16344832 &
   194543616 \\
$c= 9 $&0 &81 & 405 & 9882 & 60750 & 1108890 & 8197686 &
   136465560 \\
\end{tabular}
\caption{Values of $i_{n}^{(c)}$ for $1\leq n \leq 8$ and $1\leq c\leq 9$.}\label{tab4}
\end{table}

\begin{openproblem}
Let $b_{c}(n,k)$ denote the number of colored involutions of size $n$ with exactly $k$ inversions. 
Find an expression to calculate this sequence. Additionally, determine the generating function for the sequence \( b_c(n,k) \).\end{openproblem}


\begin{thebibliography}{99}



\bibitem{Alek}
M.~A.~Alekseyev.
\newblock Derangements with a fixed number of inversions.
\newblock International Workshop on Combinatorial Arguments, Problem Section, 2016.

\bibitem{Andrews} G.~E.~Andrews.
\newblock Euler’s pentagonal number theorem.
\newblock \emph{Math. Mag.} \textbf{56}(5) (1983), 279--284.

\bibitem{Ars1}
H.~Arslan.
\newblock A combinatorial interpretation of Mahonian numbers of type $B$.
\newblock arXiv:2404.05099, (2024).



\bibitem{Ars2} H.~Arslan, A.~Altoum, and M.~Zaarour.
\newblock An inversion statistic on the generalized symmetric groups.  
\newblock \emph{Adv. in Appl. Math.}  \textbf{154} (2024), 102655.

\bibitem{Assaf}
S.~H.~Assaf. Cyclic derangements. \emph{Electron. J. Combin.} \textbf{17} (2010), \# R163.

\bibitem{Bagno} E.~Bagno. Euler-Mahonian parameters on colored permutations groups. S\'em. Lothar. Combin. \textbf{51} (2004), Article B51f.

\bibitem{FZ}
H.~L.~M.~Faliharimalala and J.~Zeng. Fix-Euler-Mahonian statistics on wreath product. \emph{Adv. in Appl. Math.} \textbf{46} (2011), 275--295.

\bibitem{CM} R.~L.~Graham, D.~E.~Knuth, and O.~Patashnik. \textit{Concrete Mathematics: A Foundation for Computer Science}. 2nd ed., Addison-Wesley, 1994.

\bibitem{GH} Y.~Ghemit and M.~Ahmia. An analogue of Mahonian numbers and log-concavity. \emph{Ann. Comb.} \textbf{27} (2023), 895--916.


\bibitem{KN1} D.~E.~Knuth. \emph{The Art of Computer Programming}  Vol. 3. Addison-Wesley, Reading,
MA, 1998. 

\bibitem{Mac} P.~A.~MacMahon. \emph{Combinatory Analysis}, vol. 1, Cambridge Univ. Press, Cambridge, 1918.

\bibitem{Mac2} P.~A.~MacMahon. The indices of permutations and the derivation therefrom of functions of a single varaible asscoiated with the permutations of any
assemblage of objects. \emph{Amer. J. Math.} \textbf{35} (1913), 281--322.

\bibitem{Moll} I.~Mez\H{o}, V.~Moll, J.~L.~Ram\'irez, and D.  Villamizar. On the $r$-derangements of type $B$. Online J. Anal. Comb. \textbf{16} (2021), \#05.

\bibitem{MR} I.~Mez\H{o} and J.~L.~Ram\'irez. A combinatorial approach to derangement matrix of type $B$. \emph{Linear Algebra Appl.} \textbf{582} (2019), 156--180.
 
\bibitem{KN2} E.~Netto. \emph{Lehrbuch der Combinatorik}. Leipzig, B. G. Teubner, 1901.

\bibitem{OEIS}
OEIS~Foundation Inc., The {O}n-{L}ine {E}ncyclopedia of {I}nteger {S}equences,
  Published electronically at \url{https://oeis.org}, 2025.

\bibitem{Ste} E.~Steingr\'imsson. Permutation statistics of indexed permutations. \emph{Europ. J. Combin.} \textbf{15} (1994), 187--205.

\end{thebibliography}
\end{document}